\documentclass[12pt]{amsart}
\usepackage{color}
\usepackage[all,cmtip]{xy}
\usepackage{amssymb}
\usepackage{amsmath}
\usepackage{amsthm}
\usepackage{amscd}
\usepackage{amsfonts}
\usepackage{graphicx}
\usepackage{multicol}
\usepackage{appendix}

\calclayout

\newtheorem{dummy}{anything}[section]
\newtheorem{theorem}[dummy]{Theorem}

\newtheorem{lemma}[dummy]{Lemma}
\newtheorem{proposition}[dummy]{Proposition}
\newtheorem{corollary}[dummy]{Corollary}

\theoremstyle{definition}
\newtheorem{definition}[dummy]{Definition}

  \newtheorem{remark}[dummy]{Remark}
   
    \newtheorem{question}[dummy]{Question}
    
  \newtheorem*{acknowledgement}{Acknowledgement}

\newcommand{\cAb}{\mathcal Ab}

\newcommand{\cF}{\mathcal F}

\newcommand{\cH}{\mathcal H}
\newcommand{\cI}{\mathcal I}

\newcommand{\cL}{\mathcal L}

\newcommand{\cTop}{\mathcal Top}

\newcommand{\bbC}{\mathbb C}
\newcommand{\bbD}{\mathbb D}

\newcommand{\bbH}{\mathbb H}

\newcommand{\bbQ}{\mathbb Q}
\newcommand{\bbR}{\mathbb R}
\newcommand{\bbS}{\mathbb S}
\newcommand{\bbZ}{\mathbb Z}

\DeclareMathOperator{\Hom}{Hom} 
\DeclareMathOperator{\Aut}{Aut}
 
\DeclareMathOperator{\Ext}{Ext}

\DeclareMathOperator{\Res}{Res}

\DeclareMathOperator{\Map}{Map}

\DeclareMathOperator{\ind}{Ind}

 \DeclareMathOperator{\Iso}{Iso}

\newcommand{\iso}{\cong}
\newcommand{\G}{\Gamma}

\DeclareMathOperator{\Rep}{Rep}
\DeclareMathOperator{\Or}{Or}

\newcommand{\la}{\langle}
\newcommand{\ra}{\rangle}

\newcommand{\bd}{\partial}
\newcommand{\id}{\mathrm{id}}

\def\maprt#1{\smash{\,\mathop{\longrightarrow}\limits^{#1}\,}}

\begin{document}

\title{Constructing homologically trivial actions on products of spheres}

\author{\" Ozg\" un \" Unl\" u and Erg{\" u}n Yal\c c\i n }

\address{Department of Mathematics, Bilkent
University, Ankara, 06800, Turkey.}

\email{unluo@fen.bilkent.edu.tr \\
yalcine@fen.bilkent.edu.tr}

\keywords{}

\thanks{2010 {\it Mathematics Subject Classification.} Primary: 57S25; Secondary: 55R91.} \thanks{Both of the authors are partially supported by T\" UB\. ITAK-TBAG/110T712.}

\date{\today}

\begin{abstract}
We prove that if a finite group $G$ has a representation with fixity $f$, then it acts freely and homologically trivially on a finite CW-complex homotopy equivalent to a product of $f+1$ spheres. This shows, in particular, that every finite group acts freely and homologically trivially on some finite CW-complex homotopy equivalent to a product of spheres. 
\end{abstract}

\maketitle

\section{Introduction}
\label{section:Introduction}
It is known that every finite group acts freely on a product of spheres $\bbS ^{n_1}\times \cdots \times \bbS ^{n_k}$ for some $n_1, n_2, \dots, n_k$. This follows from a construction given in \cite[page 547]{oliver} which is attributed to J. Tornehave by B. Oliver. The construction is based on a simple idea that one can permute the spheres in a product to get smaller isotropy. More specifically, for a finite group $G$, one defines the $G$-space $X$ as a product of $G$-spaces $\Map _{\la g \ra } (G, \bbS (\rho_g))$ over all elements $g\in G$, where $\bbS (\rho_g)$ denotes the unit sphere of a nontrivial one-dimensional complex representation $\rho_g: \la g \ra \to \bbC$. Note that because of the way $X$ is constructed, $G$ acts freely on $X$, but the induced action of $G$ on the homology of $X$ is not a trivial action in general. It is interesting to ask if there exists a similar construction so that the induced action on homology is trivial. The following was stated as a problem by J. Davis in the Problem Session of Banff 2005 conference on Homotopy Theory and Group Actions:

\begin{question}\label{ques:jdavis} Does every finite group act freely  on some product of spheres with trivial action on homology?
\end{question}

To find a free, homologically trivial action of a finite group $G$ on a product of spheres, one may try to take a family of $G$-spheres $\bbS^{n_i}$ for $i=1,\dots, k$ and let $G$ act on the product $\bbS^{n_1}\times \cdots \times \bbS^{n_k}$ by  diagonal action.  An action which is obtained in this way is called a {\it product action} in general and a {\it linear action} if each $G$-sphere in the product is a unit sphere of a representation. By construction, product actions are homologically trivial if the action on each sphere is homologically trivial but not every finite group has a free product action on a product of spheres. For example, it is known that the alternating group $A_4$ cannot act freely on a product of spheres by a product action. This follows from a result of Oliver \cite[Theorem 1]{oliver} which says that $A_4$ does not act freely on any product of equal dimensional spheres with trivial action on homology. On the other hand, it can be shown that $A_4$ acts freely on $\bbS^2 \times \bbS^3$ with trivial action on homology.  So, one cannot answer the above question affirmatively by considering only the product actions.

The situation with $A_4$ is not an exceptional case. For example, if one searches through the characters of a finite group $G$  and tries to see when $G$ has a family of characters $\chi _1, \dots, \chi _k$ such that $G$ acts freely on $\bbS (\chi _1) \times \cdots \times \bbS (\chi _k )$, then one notices that not many groups have such a family of characters. In fact,  Urmie Ray \cite{ray} showed that finite groups which have free linear actions are very rare: If a finite group $G$ has a free linear action on a product of spheres, then all  nonabelian simple sections of $G$ are isomorphic to $A_5$ or $A_6$.

In this paper we attack the problem stated above using some more recent construction methods that were developed to study the rank conjecture. In particular, we use some ideas from our earlier papers \cite{unlu-yalcin1} and \cite{unlu-yalcin2} where we constructed free actions on products of spheres  for finite groups which have  representations with small fixity and for $p$-groups with small rank. Fixity of a representation $V$ of a finite group $G$ over a field $F$ is defined as the maximum of dimensions $\dim_{F} V^g $ over all elements $g \in G$. If $\rho : G \to U(n)$ is a faithful complex representation with fixity $f$, then $G$ acts freely on the space $X=U(n)/U(n-f-1)$. For small values of $f$, one can modify the space $X$ to obtain a free action on a product of $f+1$ spheres (see \cite{adem-davis-unlu} and \cite{unlu-yalcin1}). In this paper we improve this result to all values of fixity for actions on finite CW-complexes homotopy equivalent to product of spheres.

\begin{theorem}\label{thm:main}
Let $G$ be a finite group. If $G$ has a faithful complex representation with fixity $f$, then $G$ acts freely on a finite complex $X$ homotopy equivalent to a product of $f+1$ spheres with trivial action on homology.
\end{theorem}

We prove this theorem using a recursive method for constructing free actions. This method involves the construction of a $G$-equivariant spherical fibration $p : E \to X$ over a given finite $G$-CW-complex $X$ in each step. We require that the total space $E$ is also homotopy equivalent to a finite $G$-CW-complex and that the $G$-action on $E$ has smaller isotropy than the $G$-action on $X$. Once they are constructed, by taking fiber joins these fibrations can be replaced by $G$-fibrations which are non-equivariantly homotopy equivalent to trivial fibrations. This gives a $G$-action on a finite CW-complex $Y$ homotopy equivalent to $X \times \bbS^N$ for some $N>0$.  Using this method, after some steps, one gets a free action on a product of spheres. This method was first developed by Connolly and Prassidis \cite{connolly-prassidis} and later used also in \cite{adem-smith} and \cite{unlu-thesis}.

Since every finite group has a faithful complex representation, every finite group has a complex representation with fixity $f$ for some positive integer $f$. Hence, as a corollary of Theorem \ref{thm:main} we obtain an affirmative answer to Question \ref{ques:jdavis} in homotopy category.

\begin{corollary}\label{cor:main}
Every finite group acts freely and homologically trivially on some finite CW-complex $X$ homotopy equivalent to a product of spheres
\end{corollary}

The paper is organized as follows: In Section
\ref{sect:EquivariantFibrations}, we introduce $G$-fibrations and discuss the effects of taking fiber joins of $G$-fibrations. In Section \ref{sect:Federer}, we discuss the equivariant Federer spectral sequence introduced by M\o ller \cite{moller} and using it, we give another proof for a theorem  by M. Klaus \cite{klaus} (see Theorem \ref{thm:finiteness of homotopy}).  In Section \ref{sect:construction}, we introduce our main construction method, and finally in Section \ref{sect:mainthm}, we prove Theorem \ref{thm:main}.

\section{$G$-fibrations}
\label{sect:EquivariantFibrations}

In this section, we first give some preliminaries on $G$-fibrations and then prove some lemmas on fiber joins of $G$-fibrations. For more details on this material we refer the reader to \cite{lueck} and \cite{waner2}. Some of this material also appears in \cite{connolly-prassidis}, \cite{guclukanthesis}, \cite{klaus}, and \cite{unlu-thesis}.

\begin{definition}\label{defn:Gfibration} A $G$-fibration is a $G$-map $p: E \to B$ which satisfies the following homotopy lifting property for every $G$-space $X$: Given a commuting diagram of $G$-maps 
$$\xymatrix{
X \times \{ 0\} \ar[d] \ar[r]^-{h}
& E \ar[d]^{p}  \\
X \times I \ar[r]^-{H}
&  B,}\\ $$
there exists a $G$-map $\widetilde H: X\times I \to E$ such that $\widetilde H |_{X\times \{0\}}=h$ and $p \circ \widetilde H =H$.
\end{definition}

Given a $G$-fibration $p:E\to B$ over $B$, the isotropy subgroup $G_b\leq G$ of a point $b\in B$ acts on the fiber space $F_b:=p^{-1} (b)$. So, $F_b$ is a $G_b$-space. Let us denote the set of isotropy subgroups of the $G$-action on $B$ by $\Iso (B)$.

\begin{definition} Let $\{ F_H \}$ denote a family of $H$-spaces over all $H \in \Iso (B)$. If for every $b \in B$, the fiber space $F_b$ is $G_b$-homotopy equivalent to $F_{G_b}$, then $p:E \to B$ is said to have {\it fiber type} $\{ F_H \}$. 
\end{definition}

Note that in general a $G$-fibration does not have to have a fiber type, i.e., for $b_1, b_2 \in B$ with $G_{b_1}=G_{b_2}=H$, it may happen that $F_{b_1}$ and $F_{b_2}$ are not $H$-homotopy equivalent. But throughout the paper we only consider $G$-fibrations which have a fiber type. Observe that if $p: E \to B$ is a $G$-fibration such that $B^H$ is path connected for every $H \in \Iso(B)$, then $p$ has a fiber type since for every $b_1, b_2\in B^H$, the fiber spaces $F_{b_1}$ and $F_{b_2}$ are $H$-homotopy equivalent by a standard argument in homotopy theory. In our applications the $G$-fibrations that we construct will often satisfy this connectedness property. 

If $p : E \to B$ is a $G$-fibration with fiber type $\{ F_H \}$ such that for all $H \in \Iso (B)$ the fixed point space $B^H$ is connected, then the family  $\{ F_H\}$ satisfies a certain compatibility condition. To see this, let $H , K \in \Iso (B)$ such that $K^g \leq H$ for some $g \in G$. Then, we have $gB^H \subseteq B^K$, so by the connectedness of $B^K$, we obtain that for every $b \in B^H$, the $K$-space $gF_b$ is $K$-homotopy equivalent to $F_K$. This means that the $K$-spaces $\Res _ K g^* F_H$ and $F_K$ are $K$-homotopy equivalent for all $H\in \Iso (B)$ where $\Res _K g^* F_H$ is the space $F_H$ which is considered as a $K$-space through the map $K \to H$ defined by $k \to g^{-1} kg$. 
 
\begin{definition}\label{def:compatiblefamily} Let $\cH$ be a family of subgroups of $G$ closed under conjugation. A family of $H$-spaces $\{F_H\}$ over all $H \in \cH$ is called a {\it compatible family} of $H$-spaces if for every $H, K \in \cH$ with $K^g \leq H$ for some $g \in G$, the $K$-space $\Res _K g^* F_H$ is  $K$-homotopy equivalent to $F_K$, where $\Res _K g^* F_H$ is the the space $F_H$ considered as a $K$-space through the map $K \to H$ defined by conjugation $k \to g^{-1} k g$. 
\end{definition}

The main aim of this section is to introduce some tools for construction of $G$-fibrations with fiber type $\{ F_H\}$ for a given compatible family $\{ F_H\}$.  We first introduce some more terminology: 

Given two $G$-fibrations $p_1:E_1 \to B$ and $p_2: E_2 \to B$ over the same $G$-space $B$, a $G$-map $f: E_1\to E_2$ is called a fiber preserving map if it satisfies $p_2\circ f =p_1$. Two fiber preserving $G$-maps $f, f': E_1\to E_2$ are said to be $G$-fiber homotopic if there is a $G$-map $H: E_1\times I \to E_2$ which is fiber preserving at each $t\in I$ such that $H(x, 0)=f(x)$ and $H(x,1)=f'(x)$ for all $x\in E_1$. We say two $G$-fibrations $p_1 :E_1 \to B$ and $p_2: E_2 \to B$ are $G$-fiber homotopy equivalent if there are fiber preserving $G$-maps $f_1:E_1\to E_2$ and $f_2: E_2 \to E_1$ such that $f_1\circ f_2$ and $f_2 \circ f_1$ are $G$-fiber homotopic to identity maps. 

For an $H$-space $F_H$, let $\Aut _H (F_H)$ denote the topological monoid of self $H$-homotopy equivalences of $F_H$. Note that $\Aut _H (F_H)$ is not a connected space in general   but it is easy to show that all its components have the same homotopy type. When we need to choose a component, we often take the connected component which includes the identity map. We denote this component by $\Aut ^I _H (F_H)$. 

Since $\Aut _H (F_H)$ is a monoid, the usual construction of classifying spaces for monoids applies, and we get a universal fibration $E\Aut _H (F_H) \to B\Aut _H (F_H)$ with fiber $\Aut_H (F_H)$. From this one also obtains a fibration $$ F_H \to E_H \to B\Aut _H (F_H)$$ where $E_H =E\Aut_H (F_H) \times _{\Aut_H (F_H)} F_H$. This is actually an $H$-fibration with trivial  $H$-action on the base space. It turns out that this fibration is a universal fibration for all $H$-equivariant fibrations with trivial action on the base space.

\begin{theorem}\label{thm:classification} Let $H$ be a finite group, $F_H$ be a finite $H$-CW-complex, and $B$ be a CW-complex with trivial $H$-action. Then, there is a one-to-one correspondence between $H$-fiber homotopy classes of $H$-fibrations over $B$ with fiber $F_H$ and the set of homotopy classes of maps $B \to B\Aut _H (F_H )$. The correspondence is given by taking the pullback of the universal $H$-fibration described above via the map $f:B \to B\Aut_H(F_H)$.
\end{theorem}

This theorem is proved in \cite{guclukanthesis} in full detail. The proof is based on the proof of Stasheff's theorem on the classification of Hurewicz fibrations \cite{stasheff}. More general versions of this theorem also appear in \cite{french} and \cite{waner2}. 

Also note that, as in the case of orientable vector bundle theory, we can give an orientable version of the classification of $H$-fibrations over a trivial $H$-space $B$.  If $p:E \to B$ is an $H$-fibration with fiber $F_H$ over a trivial $H$-space, then there is a natural group homomorphism $\chi: \pi_1 (B)\to \pi_0 (\Aut _H (F_H))=\mathcal{E}_H (F_H)$ where $\mathcal{E}_H(F_H)$ denotes the group of homotopy classes of self $H$-homotopy equivalences of $F_H$. If this homomorphism is trivial, then we call the $H$-fibration $p$ a {\it homotopy orientable} $H$-fibration. This notion of orientability is stronger than usual notion of orientable fibration where one only requires the action on the homology of $F_H$ to be trivial (see \cite{ehrlich}).  Note that an $H$-fibration is homotopy orientable if and only if its classifying map $f: B \to B\Aut _H (F_H)$ lifts to a map 
$\tilde f : B \to B \Aut ^{I} _H (F_H )$. Also note that homotopy orientable fibrations are classified (as homotopy orientable fibrations) by the homotopy classes of maps $[B , B\Aut _H ^I (F_H)]$. We will be using these facts later in the proof of Lemma \ref{lem:jointrivial}.
 
In the rest of this section we focus on the fiber join construction performed on $G$-fibrations. We use fiber joins to kill obstructions
that occur in the construction of $G$-fibrations. The fiber join of two $G$-fibrations is defined in the following way:
Let $p_1: E_1\to B$ and $p_2: E_2\to B$ be two fibrations. We define a
fibration $E_1\times _B E_2$ and maps
$E_1\times _B E_2\to E_i$ for $i=1,2$ by the following pullback diagram
$$\xymatrix{
E_1\times _B E_2 \ar[r]\ar[d] & E_{1}  \ar[d]^{p_{1}} \\
E_2 \ar[r]^{p_2}             &  B . }$$
Then the $G$-space $E_1*_B E_2$ is defined as the homotopy pushout of the following diagram
$$\xymatrix{
E_1\times _B E_2 \ar[r]\ar[d] & E_{1}  \ar[d] \\
E_2 \ar[r] & E_1*_B E_2.  }$$
By the universal property of homotopy pushouts we get a $G$-fibration $$p_1*p_2:E_1*_B E_2\to B$$ called
the {\it fiber join of $p_1$ and $p_2$}.
Iterating this construction, we obtain a $G$-fibration $$\underset{k}{\ast} \, p: \underbrace{E*_B E*_B
\dots *_B E}_{k\text{-many}}\to B$$ which we call the $k$-fold
fiber join of $p$ with itself. Note that if $p$ is a $G$-fibration with fiber type $\{F_H\}$, then the fiber type of the $k$-fold join $\ast _k p$ is $\{\ast _k F_H\}$. If $B$ has trivial $H$-action, then the $k$-fold join is classified by a map $B \to B \Aut _H (\ast _k F_H)$. We would like to explain this map in terms of the classifying map of $p$. For this, observe  that there is a monoid homomorphism 
$$\varphi : \Aut_H (F_H) \times \cdots \times \Aut _H (F_H)  \to  \Aut_H (\ast_k F_H)$$ defined by 
 $$\varphi (a_1,\dots , a_k) (x_1 t_1, \dots ,x_k t_k)=(a_1(x_1) t_1 , \dots , a_k (x_k) t_k )$$ for every $x_1, \dots,  x_k \in F_H$ and $t_1, \dots , t_k \in [0,1]$ with $\sum_i t_i=1$. We have the following lemma:

\begin{lemma}\label{lem:fiberjoins} Let $H$ be a finite group, $F_H$ be a finite $H$-CW-complex, and $B$ be a CW-complex with trivial $H$-action. If $p: E\to B$ is a $H$-fibration with fiber type $F_H$ whose classifying map is $f: B \to B\Aut_H (F_H)$, then the classifying map of the $G$-fibration $\ast _k \, p$ is given by the composition 
\[\xymatrix@C=3pc{B \ar[r]^-{f\times \cdots \times f} & B\Aut_H(F_H)\times \cdots \times B\Aut _H (F_H) \ar[r]^-{B\varphi} & B \Aut_H (\ast _k F_H )}\]
where $B\varphi$ is the map induced from the monoid homomorphism $\varphi$ defined above. 
\end{lemma}

\begin{proof} Let $A=\prod _{i=1}^k \Aut_H  (F_H)$. By standard properties of homotopy pushout diagrams, we observe that the fibration $\ast _k \, p$ is the pullback fibration of the fibration $$ q: EA \times _A (\ast _k F_H ) \to BA $$ via the map $\prod _i  f: B \to BA$. Note that $A$ acts on $\ast_k F_H$ via the monoid homomorphism $\varphi$, so the classifying map of $q$ is $B\varphi$. This completes the proof.
\end{proof}

A special case of a $G$-fibration is a $G$-fiber bundle over a $G$-CW-complex. More specifically, if $\xi : E\to B$ is a complex $G$-vector bundle over a $G$-CW-complex $B$, then the sphere bundle $p: S(E) \to B$ of this vector bundle is a spherical $G$-fibration. Note that for every $b \in B$, the fiber space $p^{-1} (b)$ is a $G_b$-space which is homeomorphic to $\bbS (V_{G_b})$ where $V_{G_b}$ denotes vector space $\xi ^{-1}(b)$ with the induced $G_b$-action.   

Note that when $B^H$ is path connected for all $H \in \Iso (B)$, the family of complex representations $\{ V_H\}$ defined over all $H \in \Iso (B)$ is a compatible family. The compatibility of a family of representations is defined in the following way: A family of representations $\alpha_H :H \to U(n)$ over $H \in \cH$ is called a {\it compatible family of representations} if for every map $c_g: H\to K$ 
induced by conjugation with $g\in G$, 
there exists a $\gamma \in U(n) $ such that the following diagram commutes
$$\xymatrix{K \ar[d]^{c_g} \ar[r]^-{\alpha _K}
& U(n) \ar[d]^{ c_{\gamma }}  \\
H  \ar[r]^-{\alpha _H}
&  U(n) .}\\ $$
Note that if $F_H$ is an $H$-space which is $H$-homotopy equivalent to $\bbS (V_H)$ for some compatible family of complex $H$-representation $V_H$, then $\{ F_H \}$ is a compatible family of $H$-spaces. So, the sphere bundle $p: S(E) \to B$ of a $G$-vector bundle is a spherical $G$-fibration with fiber type $\{ \bbS (V_H) \}$. Note also that for every $k\geq 1$, the fiber join $\ast _k \bbS (V_H)$ is $H$-homotopy equivalent to the $H$-space $\bbS (V_H ^{\oplus k})$ where $V_H ^{\oplus k}$ denotes the $k$-fold direct sum of $V_H$.

In Section \ref{sect:construction}, we construct $G$-fibrations with fiber types of the form $\{ \bbS (V_H)\}$. The following result is used in those constructions.

\begin{lemma}\label{lem:joinswap} Let $H$ be a finite group, $F_H$ be an $H$-space which is $H$-homotopy equivalent to $\bbS (V_H)$ for some complex $H$-representation $V_H$. Let $\gamma , \gamma ^1 : \Aut _H (F_H)\to \Aut _H (\ast _k F_H)$ be maps defined by $\gamma (a)= \varphi (a,a, \dots, a)$ and $\gamma ^1 (a)= \varphi (a, \id , \dots, \id)$, respectively. Then, the induced group homomorphisms $\gamma _*$ and $\gamma ^1 _*$ on homotopy $\pi_q (\Aut _H (F_H)) \to \pi _q (\Aut _H (\ast _k F_H)) $ satisfy the relation $\gamma _*=k \gamma^1 _*$.
\end{lemma}

\begin{proof} Let $\gamma ^i: \Aut _H (F_H)\to \Aut _H (\ast _k F_H )$ be the map defined by $$\gamma ^i (a)=\varphi (\id, \dots, a, \dots, \id)$$ where $a$ is on the $i$-th coordinate. We have $\gamma=\gamma ^1 \gamma ^2 \cdots \gamma ^k$ under the product induced by the product in the monoid $A$. Since the group operation on $\pi _q (\Aut _H (\ast _k F_H))$ coming from the monoid structure on $\Aut _H (\ast _k F_H)$ coincides with the usual group structure on homotopy groups, we have 
$\gamma _*= \gamma ^1 _* +\cdots + \gamma ^k _*$. So, to complete the proof, it is enough to show that $\gamma ^i$ and $\gamma ^j$ are homotopic for every $i,j \in \{ 1, \dots, k\}$.
Since $F_H$ is $H$-homotopy equivalent to $\bbS (V_H)$, it is enough to prove this for $\bbS(V_H)$. Note that in this case
we have $\gamma ^i=T(i,j) \gamma ^j$ where $T(i,j): V_H^{\oplus k} \to V_H^{\oplus k}$ is a linear transformation which swaps the $j$-th summand with the $i$-th summand. Since $U(n)$ is connected, there is a path between $T(i,j)$ and the identity. Using this path, we can define a homotopy between $\gamma ^i$ and $\gamma ^j$.
\end{proof}

\begin{remark} For more general $H$-spaces $F_H$, there exists a swap map $$S(i,j): \ast _k F_H \to \ast _k F_H,$$ which swaps the $i$-th and $j$-th coordinates, similar to  the linear transformation $T(i,j)$ in the proof of Lemma \ref{lem:joinswap}. If $F_H$ is a free $H$-space homotopy equivalent to an odd dimensional sphere, then $S(i,j)$ will be homotopy equivalent to the identity map. If the $H$-action on $F_H$ is not free, then the swap map $S(i,j)$ is not homotopy equivalent to the identity in general even when $F_H$ is homotopy equivalent to an odd dimensional sphere. On the other hand, if $F_H$ is a homotopy representation with the property that all fixed point spheres are odd dimensional, then
under certain conditions on $H$ or on the dimension function of $F_H$, one can prove that $S(i,j)$ is homotopy equivalent to the identity (see Proposition 20.12 in \cite{lueck}).  
\end{remark}

We end this section with the following observation.

\begin{lemma}\label{lem:jointrivial} Let $H$ be a finite group and $p: E \to \bbS ^n $ be an $H$-fibration over the trivial $H$-space $\bbS ^n$ where $n\geq 2$. Suppose that the fiber type $F_H$ of $p$ is $H$-homotopy equivalent to $\bbS (V_H)$ for some complex $H$-representation $V_H$. If \ $\pi _{n-1} ( \Aut _H (F_H ))$ is a finite group of order $N$, then $\ast _N p$ is $H$-fiber homotopy equivalent to the trivial fibration. 
\end{lemma}

\begin{proof} By Theorem \ref{thm:classification}, the $H$-fibration $p$ is classified by the homotopy class of a map $f: \bbS^n \to B\Aut _H (F_H)$. Since $n\geq 2$, this map lifts to map $\widetilde f : S^n \to B\Aut^{I}_H (F_H)$, so we can assume that the $H$-fibration $p$ is an homotopy orientable fibration . Since $B\Aut ^I _H (F_H)$ is simply connected, we have $$[S^n, B\Aut^I _H (F_H )]\cong \pi _n (B\Aut ^I _H (F_H))\cong \pi_{n-1} (\Aut _H (F_H)).$$ So, $p$ is classified by a homotopy class $\alpha\in \pi _{n-1} (\Aut _H (F_H))$. By a slightly modified version of Lemma \ref{lem:fiberjoins}, it is easy to see that the fiber join $\ast _N p$ is classified by $\gamma _* (\alpha)$ where $\gamma: \Aut _H (F_H) \to \Aut _H ( \ast _H F_H )$ is the map defined by $\gamma (a)=\varphi (a,\dots, a)$. By Lemma \ref{lem:joinswap}, we have $$\gamma _* (\alpha)=N\gamma _* ^1 (\alpha)=\gamma _* ^1 (N \alpha )=0.$$
So, $\ast _N p$ is $H$-fiber homotopy equivalent to the trivial fibration.
\end{proof}

\section{Equivariant Federer spectral sequence}
\label{sect:Federer}

The main purpose of this section is to prove the following theorem which is due to M. Klaus \cite{klaus}. We give a different proof here using the equivariant Federer spectral sequence which was introduced by M\o ller in \cite{moller}.

\begin{theorem}[Klaus \cite{klaus}]\label{thm:finiteness of homotopy} Let $G$ be a finite group and $V$ be a complex representation of $G$. Then, for every $n>0$, there is an $m\geq 1$ such that $\pi_n (\Aut _G (\bbS (V ^{\oplus k})) $ is finite for all $k \geq m$.
\end{theorem}

Before the proof, we first introduce the standard definitions about Bredon cohomology and local coefficients systems
that we are going to use to describe the equivariant Federer spectral sequence. For more details on Bredon cohomology we refer the reader to \cite{bredon}.
  
Let $X$ be a topological space.  A {\it local coefficient system} of $X$ is a functor $L:\Pi (X)\to \cAb $ where $\Pi (X)$ denotes the fundamental groupoid of $X$ and
$\cAb $ denotes the category of abelian groups. Let $\cL $ denote the category whose objects are pairs $(X,L)$ where $X$ is a topological space and $L$ is a local coefficient system of $X$ and a morphism from $(X,L)$ to $(Y,M)$ is a pair $(f,\varphi )$ where $f:X\to Y$ is a continuous function and $\varphi $ is a natural transformation from $L$ to $M\circ f_*$. Here $f_*$ denotes the functor $f_*:\Pi (X)\to\Pi (Y)$ induced by $f$ sending $x$ to $f(x)$ and $\gamma $ to $f\circ \gamma $. 

Let $G$ be a finite group and $\Or _{G}$ denote the {\it orbit category} of $G$ whose objects are
orbits $G/H$ where $H$ is a subgroup of $G$. The morphisms of $\Or _{G}$ from $G/H$ to $G/K$ are $G$-maps between them where we consider the left cosets $G/H$ and $G/K$ as left $G$-sets. We denote the morphism from $G/H$ to $G/K$ which sends $H$ to $aK$ by $\hat{a}$.  

\begin{definition}\label{defn:eqlocalcoefsystem} Let $\cTop $ denote the category of topological spaces and continuous maps. Let $X$ be a $G$-space. We define a contravariant functor $\Phi(X):\Or _{G}\to \cTop $ which sends $G/H$ to $X^H$ and $\hat{a}$ to $a:X^K\to X^H$. A $G$-{\it equivariant local coefficient system} on $X$ is a contravariant functor $\underline{\, L} :\Or _{G}\to \cL $ such that $F\circ  \underline{\, L} =\Phi(X)$ where $F:\cL \to \cTop$ is the forgetful functor  which sends $(X,L)$ to $X$ and $(f,\varphi )$ to $f$. We will use the following notation: $\underline{\, L}(G/H)=(X^H , L(G/H))$.
\end{definition}

Let $\underline{L}$ be a $G$-equivariant local coefficient system of a finite $G$-CW-complex $X$. Recall that for a coefficient system $L$ on $X$, the group of $n$-cochains
$\Gamma ^n(X; L)$ is defined as the group of all functions $c$ which take $n$-cells $\sigma $ in $X$ with characteristic map $h_{\sigma}:\Delta ^n\to X$ and send it to an element $c(\sigma )$ in $L(z_{\sigma })$ where $\Delta ^n$ is considered as the convex hull of a linearly independent subset $\{e_0,e_1,\dots e_n\}$ of $\bbR^{n+1}$ and  $z_{\sigma }=h_{\sigma}(e_0)$. The coboundary operator $\delta:\Gamma ^{n-1}(X;L)\to \Gamma ^{n}(X;L)$ is defined as follows: For $c$ in $\Gamma ^{n-1}(X;L)$ and $\sigma $ a $n$-cell  in $X$, we have
$$(-1)^n(\delta c)(\sigma )=L(\gamma _\sigma )^{-1}c(\partial _0\sigma )+\sum _{i=1}^{n} (-1)^{i}c(\partial _i\sigma )\in L(z_{\sigma })$$
where  $\partial _i\sigma $ is the $(n-1)$-cell with characteristic map $h_{\sigma}\circ d^n_i$ and $\gamma _\sigma (t)=h_{\sigma}((1-t)e_1+te_0)$. Here $d^n_i:\Delta ^{(n-1)}\to \Delta ^n$ is the affine map sending $e_j$ to $e_j$ when $j<i$ and to $e_{j+1}$ otherwise. Now for the $G$-equivariant coefficient system $\underline{L}$, we define $\Gamma ^n_G(X;\underline{L} )$ as elements in the direct sum $$\underline{\, c}=(c(G/H)) \in \bigoplus _{H\leq G}\Gamma ^n(X^H;L(G/H) )$$ which satisfy the following condition:
$$c(G/H)(a\sigma )=\underline{\, L}(\hat{a})(z_{\sigma })(\, c(G/K)(\sigma )\, ) \text{\ \  in } L(G/H)(az_{\sigma })$$
for all $\sigma \in X^K$ and $a\in G$ with $a^{-1}Ha \leq K$.   We can define a coboundary operator $\underline{\delta} :\Gamma ^{n-1}_G(X,\underline{L})\to \Gamma ^{n}_G(X,\underline{L})$ by the direct sum of the ordinary coboundary operators
$$\underline{\, \delta }=\bigoplus _{H\leq G} \delta _H:\bigoplus _{H\leq G}\Gamma ^n(X^H ;L(G/H) )\to\bigoplus _{H\leq G}\Gamma ^{n+1}(X^H ; L(G/H) )$$
Since the ordinary coboundary operator is a natural transformation from $\Gamma^n$ to $\Gamma^{n+1}$ considered as functors from $\cL $ to $\cAb $, we get $\underline{\, \delta }(\Gamma ^n_G(X;\underline{\, L} ))\subseteq \Gamma ^{n+1}_G(X;\underline{\, L} )$. The above definition easily generalizes to relative $G$-CW-complexes. 

\begin{definition} The $n$-th cohomology of a relative $G$-CW-complex $(X,A)$ with $G$-equivariant local coefficients $\underline{\, L}$ is defined as follows: $$H ^n_G(X,A;\underline{\, L} )=H^n\left(\Gamma ^*_G(X,A;\underline{\, L} ),\underline{\, \delta }\right).$$ 
\end{definition}

We will be using the Bredon cohomology with a particular local coefficient system that comes from a $G$-fibration. We now introduce this coefficient system. 

Let $p:E\to B$ be a $G$-fibration. Then $p^H:E^H\to B^H$ is a fibration for all $H\leq G$. Assume that for all $H\leq G$ and for all $b\in B$, the space $p^{-1}(b)^H$ is a path connected simple space.   Associated to the fibration $p$, there is a $G$-equivariant local coefficient system on the $G$-space $B$ defined as the functor $$\pi _n (\cF ): \Or _{G}\to \cL $$ which sends $G/H$ to $(B^H, \pi _n (\cF ^H))$ and $\hat{a}$  to $(a,a_*)$ where $\pi _n (\cF ^H):\Pi(B^H)\to \cAb $ is the functor which sends $b \in B$ to $\pi _n (p^{-1}(b)^H)$ and  sends a path $\gamma $ with $\gamma(0)=c$ and $\gamma(1)=b$ to a homomorphism from $\pi _n (p^{-1}(b)^H)$ to $\pi _n (p^{-1}(c)^H)$ which is induced by a map admissible over $\gamma $ (see \cite[page 185]{whitehead}). 

Let $(X,A)$ be a finite $G$-CW-complex, $p:E\to B$ be a $G$-fibration, and $u:X\to E$ be a $G$-equivariant map. As above,  assume that for all $H\leq G$ and for all $b\in B$, the space $p^{-1}(b)^H$ is a path connected simple space, and let $\pi _q (\cF )$ be the $G$-equivariant local coefficient system on $B$ which was introduced above. By abuse of notation, we will again write $\pi _q (\cF )$ for the $G$-equivariant local coefficients system on $X$ induced from $\pi _q (\cF )$ via the map $p\circ u$. Let $F_u(X,A;E,B)^G$ denote the space of all equivariant maps $v:X\to E$ such that $v|_A=u|_A$ and $p\circ v=p\circ u$ with compact open topology.  

\begin{theorem}[M\o ller \cite{moller}]\label{thm:equivFedererSS}   There is a spectral sequence with $E^2$-term 
$$E^2_{pq}=H^{-p}_G(X,A;\pi _q (\cF ))$$
for $p+q\geq 0$ and $E^2_{pq}=0$ otherwise, converging to
$\pi _{p+q}(F_u(X,A;E,B)^G,u)$
when $p+q>0$.
\end{theorem}

The spectral sequence above is called the equivariant Federer spectral sequence since it is the equivariant version of a spectral sequence introduced by Federer \cite{federer}. We will be using this spectral sequence for the following special case: Let $X$ be a finite $G$-CW-complex such that $X^H$ is a path connected simple space for all $H \leq G$. Take $A=\emptyset $, $E=X$, $B=*$, $p:E\to B$ to be the constant map, and $u:X\to E$ to be the identity map. Then $F_{\id} (X, \emptyset ; X , * )$ will be homotopy equivalent to the identity component of $\Aut _G(X)$. Since all the components of $\Aut _G(X)$ have the same homotopy type, we have
$$\pi _n (\Aut _G(X))\iso \pi _n (F_{\id} (A, \emptyset; X, *))$$ for all $n>0$. So, we can use the equivariant Federer spectral sequence to calculate the homotopy groups of $\Aut _G (X)$.

Also note that in the situation we consider, the local coefficient system is constant on orbits. Bredon cohomology with coefficients in a $G$-equivariant local coefficients system has an alternative description when the 
coefficient system is constant on $G$-orbits. This description involves the modules over the orbit category which we will define now.

Let $G$ be a finite group and let $\Gamma $ denote the orbit category $\Or _G$. A functor 
from $\Or _G$ to the category of abelian groups $\cAb$ is 
called an $\bbZ \G$-module. Morphisms between $\bbZ \G$-modules are given by natural transformations. Given a $G$-CW-complex $X$, we define a chain complex of $\bbZ \G$-modules by taking $C_n(X^?)$ as the functor $\Or _G\to \cAb $ which sends $G/H$ to the $n$-th cellular chains $C_n(X^H)$ and sends $\hat{a}:G/H\to G/K$ to the group homomorphism $a_*:C_n(X^K)\to C_n(X^H)$. 

Let $\underline{L}$ be a $G$-equivariant local coefficient system of $X$. Suppose that there exists a $\bbZ \G$-module $M$ such that $L(G/H)(x)=M(G/H)$ and $L(G/H)(\gamma )=\id _{M(G/H)}$ for all $x\in X^H$ and all paths $\gamma $ in $X^H$. Then we have $$\Gamma ^n_G(X;\underline{L} )\cong \Hom _{\bbZ \G}(C_n(X^?),M)$$ where the isomorphism is given by sending $\underline{\, c}=(c(G/H))$ in $\Gamma ^n_G(X; \underline{L} )$ to the homomorphism $\alpha : C_n(X^?) \to M$ defined by   $\alpha (G/H)(\sigma )= c(G/H)(\sigma )$ for all $H\leq G$. 
The boundary maps at each $H$ are compatible with respect to inclusions and conjugations, so they combine together to give a $\bbZ \G$-module map  $\bd : C_n (X^?) \to C_{n-1} (X^{?})$ for every $n$. Using these boundary maps, we obtain a cochain complex of abelian groups
$$C^n (X, M) =\Hom _{\bbZ \G}(C_n(X^?),M) \cong \bigoplus _{[\sigma ]\in \cI _n}M(G/G_{\sigma })$$
where $\cI _n$ is a set of $G$-orbits of $n$-cells in $X$. Note that the last isomorphism comes from the standard properties of free $\bbZ \G$-modules (see \cite[Sec. 9]{lueck}). The cohomology of this cochain complex is denoted by $H^n _G (G, M)$ and we have an isomorphism $H^n _G (X; \underline{L}) \cong H^n _G (X; M)$ for all $n\geq 0$ when 
$\underline{L}$ is a $G$-equivariant local coefficients system on $X$ and $M$ is a $\bbZ \G$-module such that $M(G/H)=\underline L (G/H)(x)$ for all $H \leq G$ and $x \in X^H$.
In our situation, this gives an isomorphism 
$$H ^n _G(X; \pi _q (\cF ) )\cong H ^n _G(X; \pi _q (X^? ) )$$
where $\pi _q (X^?)$ is the $\bbZ \G$-module  $\pi_q(X^?):\Or _G\to \cAb $ which sends $G/H$ to $\pi _q(X^H)$ and sends $\hat{a}:G/H\to G/K$ to $a_*:\pi _q (X^K)\to \pi _q (X^H)$.  Also note that on the cochain level, we have 
$$C^n (X, \pi _q(X^?))=\Hom _{\bbZ \G}(C_n(X^?),\pi _q(X^?))\cong \bigoplus _{[\sigma ]\in \cI _n}\pi _q(X^{G_\sigma }).$$
So, we have an explicit description of the $E_{pq}^2$-terms of the equivariant Federer spectral sequence. Now we are ready to prove the main theorem of this section.

\begin{proof}[Proof of Theorem \ref{thm:finiteness of homotopy}] Let $G$ be a finite group and $X$ be a finite $G$-CW-complex which is $G$-homotopy equivalent to $\bbS (V)$ for some complex representation $V$ of $G$. In fact, we only need $X$ to be a $G$-homotopy representation with odd dimensional fixed point spheres for our arguments to work (see \cite[pg. 392]{lueck} for a definition of homotopy representation). Let $n$ be a fixed positive integer. We want to show that there is an $m\geq 1$ such that $\pi_n (\Aut _G (\ast _k X))$ is finite for all $k \geq m$. Let $X_k=*_k X$ denote the $k$-fold join of $X$. By Theorem \ref{thm:equivFedererSS}, there is a spectral sequence with
$$E^2_{pq}=H^{-p}_G(X_k;\pi _q (X_k^?))$$
for $p+q\geq 0$ and $E^2_{pq}=0$ otherwise, converging to
$\pi _{p+q}(\Aut_G(X_k))$ when $p+q>0$. Since $X_k$ is finite dimensional, to show that $\pi _n (\Aut _G (X_k))$ is finite it is enough to show $H^{-p}_G(X_k;\pi _q (X_k^?))$ is finite for every pair $(p, q)$ with $p+q = n$. Note that for this we need to show that there is an $m \geq 1$ such that for all $k\geq m$, the cohomology group $H^{q-n}_G(X_k;\pi _q (X_k ^?)\otimes \bbQ)$ is zero for all  $q \geq n$.

Let $\{ n_1, n_2, \dots, n_s\}$ be the set of all distinct dimensions of fixed subspaces $V^H$ over all subgroups $H \leq G$. Assume that $n_1< n_2 < \dots < n_s$. Note that the fixed point spheres $X_k^H$ have dimensions $\{kn_i-1 \ |\ i=1,\dots, s\}.$ Since homotopy groups $\pi _i (S^{2j-1}) $ of an odd dimensional sphere are all finite except when $i=2j-1$, we have $\pi _ q (X_k^?) \otimes \bbQ =0$ for all $q$ which is not equal to $kn_i-1 $ for some $i$. If $q=kn_i-1$ for some $i$, then we have 
$$ H_G ^{q-n} (X_k ; \pi _q (X_k^{?} )\otimes \bbQ )=H_G ^{kn_i-n-1 } (X_k; M_i) $$
where $M_i$ is the $\bbZ\G$-module such that $M_i (H)=\bbQ$ for all subgroups $H \leq G$ satisfying $\dim V^H=n_i$. To complete the proof we need to show that this cohomology group is zero for all $i\in \{ 1, \dots, s\}$. 

Note that there is a well-known first quadrant spectral sequence with $E_2$-term
$$ E_2 ^{pq}=\Ext _{\bbZ \G } ^p (H_q (X_k^?), M_i )$$ which converges to $H_G ^{p+q} (X_k; M_i)$ (see \cite[Prop. 3.3]{unlu-yalcin2}). Since the coefficient module $M_i$ takes only the values $\bbQ$, we can replace $H_p(X_k^?)$ with $H_p (X_k^? ; \bbQ)$ and take the ext-groups over $\bbQ \G$. Note that the $\bbQ \G$-module $H_{p} (X_k^?; \bbQ)$ is zero at all dimensions except when $p=kn_i-1$ for some $i$. Let $N_i$ denote the $\bbQ \G$-module $H_{kn_i-1} (X_k^?; \bbQ)$ for all $i=1,\dots ,s$. To prove that $H_G ^{kn_i -n-1} (X_k; M_i)=0$ for all $i$, it is enough to show that the ext-group $$\Ext _{\bbQ \G} ^{k(n_i-n_j)-n} ( N_j, M_i)$$
is zero for all $j\leq i-1$. Let $l_j $ denote the length of the $\bbQ \G$-module $N_j$ for every $j$ (see \cite[pg. 325]{lueck} for a definition). Then, by \cite[prop. 17.31]{lueck}, the above ext-group is zero if $k(n_i-n_j)-n \geq l_j$. Let $l=\max_j \{l_j\}$. Then if $k\geq n+l$, then the above inequality will hold for every $j\leq i-1$. This completes the proof.
\end{proof}

\section{Construction of spherical $G$-fibrations}
\label{sect:construction}

We start with proving a proposition which is an important tool for constructing $G$-fibrations. In different forms, this proposition also appears in \cite{connolly-prassidis}, \cite{klaus}, and \cite{unlu-thesis}. Here we give a proof of it for completeness since it is the main 
ingredient in the proof of Theorem \ref{thm:main}.

\begin{proposition}\label{pro:mainconsttool} Let $G$ be a finite group, $B$ be a $G$-CW-complex, and let $\{V_H\}$ be a compatible family of complex representations over all $H \in \Iso (B)$. Let $q_n : E_n \to B^{(n)} $, $n\geq 2$, be a $G$-fibration with fiber type $\{ \bbS(V_H )\}$ where $B^{(n)}$ denotes the $n$-skeleton of $B$. Then there is an integer $k\geq 1$ and a $G$-fibration $q_{n+1}:E_{n+1}\to B^{(n+1)}$ such that the restriction of $q_{n+1}$ to $B^{(n)}$ is $G$-fiber homotopy equivalent to $\ast _k q_n$. In particular, the fiber type of $q_{n+1}$ is  $\{ \bbS (V_H ^{\oplus k} )\}$.
\end{proposition}

\begin{proof} By the definition of $G$-CW-complexes, there exists a pushout diagram
\begin{equation}\label{eqn:pushout}
\vcenter{\xymatrix{\displaystyle \coprod _{i\in I_{n+1}} G/H_i\times \bbS ^{n}
\ar[r]^-{\coprod   f_i}\ar[d] & B^{(n)}  \ar[d] \\
\displaystyle \coprod _{i\in I_{n+1}}G/H_i\times \bbD ^{n+1}
\ar[r]^-{\coprod g_i} & B^{(n+1)}}}\end{equation}where $I_{n+1}$ is an
indexing set of orbits of $(n+1)$-cells in $B$. For each $i\in I_{n+1}$, let $q_{n,i}$ denote the $G$-fibration obtained by 
the following pullback diagram
$$\xymatrix{ E_{n, i} \ar[r]\ar[d]^{q_{n,i}} & E_n  \ar[d] \ar[d]^{q_n} \\G/H_i\times \bbS ^{n}
\ar[r]^-{f_i}& B^{(n)}\ . }$$ Restricting $q_{n,i}$ to the sphere $\bbS ^n$ in $G/H_i \times \bbS ^n$ which is fixed by $H_i$, we obtain an $H_i$-fibration $q_{n,i} |_{\bbS^n} : q_{n,i} ^{-1} (\bbS ^n ) \to \bbS ^n $ such that the $H_i$-action on the base space is trivial. By Theorem \ref{thm:classification} and by the argument in the proof of Lemma \ref{lem:jointrivial}, such a fibration is classified by 
a homotopy class $\alpha_i \in \pi _{n-1} (\Aut _{H_i} (\bbS(V_{H_i}) ))$.

By Theorem \ref{thm:finiteness of homotopy}, for each $H \in \Iso (B)$, there is an $m_{H} \geq 1$ such that $\pi _{n-1} (\Aut _{H} (\bbS (V_{H} ^{\oplus k} ))$ is finite for all $k \geq m_H$. Let $m=\max \{m_H \, |\, H \in \Iso (B)\}$. Then the group $$\pi _{n-1} (\Aut _{H} (\bbS (V_{H} ^{\oplus  m } ))$$ has finite order, say $d_H$, for all $H \in \Iso (B)$. Let $d=\prod _H d_H$. By Lemma \ref{lem:jointrivial}, the $H_i$-fibration $\ast _{dm} (q_{n,i} |_{\bbS ^n })$ is $H_i$-fiber homotopy equivalent to the trivial fibration for all $i \in I_{n+1}$. This implies that the $G$-fibration $p$  obtained by  the following pullback diagram
$$\xymatrix{
W \ar[r]^{f}\ar[d]^{p} & \ast _{dm} E_n  \ar[d]^{\underset{dm}{\ast}\, q_n} \\ \displaystyle \coprod _{i\in I_{n+1}}G/H_i\times \bbS ^{n}\ar[r]^-{\coprod   f_i} & B^{(n)}  }$$
is $G$-homotopy equivalent to the trivial fibration. Let $$\varphi :  \coprod _{i \in I_{n+1}} G \times _{H_i} \bbS (V_{H_i} ^{\oplus dm})\times \bbS ^n  \to W $$ be a $G$-fiber homotopy equivalence between the trivial fibration and $p$. We can use $\varphi$ to glue the cone of the trivial fibration and obtain a quasifibration
$$\xymatrix{\left( \displaystyle \coprod _{i\in I_{n+1}}G\times _{H_i} \bbS (V_{H_i} ^{\oplus dm})\times \bbD ^{n+1} \right)  \cup _{f \circ \varphi}\left( \underset{dm}{*}E_n\right)  \ar[rr] & &  B^{(n+1)} }.$$
There is a construction called gammafication that converts a quasifibration to a fibration and this construction also works for $G$-quasifibrations (see \cite[pg. 375]{waner1}). Applying gammafication to the above $G$-quasifibration, we obtain a $G$-equivariant spherical fibration $q_{n+1}:E_{n+1}\to B^{(n+1)}$ whose fiber type is $\{ \bbS (V_H ^{\oplus dm})\}$.
\end{proof}

\begin{remark} Another possible way of completing the final step of the above construction is to attach trivial $G$-fibrations over $(n+1)$-cells with the space $\ast _{dm} E_n$ using $G$-tubes (see \cite[Theorem 3.1]{guclukanpaper}). When these $G$-tubes are used one does not need the gammafication construction since one directly gets $G$-fibrations. This method is explained in detail in \cite{guclukanthesis} and \cite{guclukanpaper}.
\end{remark}

As a corollary of Proposition \ref{pro:mainconsttool}, we obtain the following which is also proved in \cite{klaus} as Proposition 2.7.

\begin{proposition}[Proposition 2.7, \cite{klaus}] Let $G$ be a finite group and $B$ be a finite dimensional $G$-CW-complex. Let $\{ V_H\}$ be a compatible family of complex representations over all $H \in \Iso (B)$. Then there exists an integer $k \geq 1$ and a $G$-equivariant spherical fibration $q:E\to B$ such that the fiber type of $q$ is $\{ \bbS (V_H ^{\oplus k} ) \}$.
\end{proposition}

\begin{proof} Let $${\bf A}=(\rho_H) \in \lim _{\underset{H \in \cH}{\longleftarrow}} \Rep (H, U(n))$$ 
where $\cH =\Iso (B)$ and $\rho_H$ is a representation for $V_H$ for every $H \in \cH$. Let $q: E_{\cH } (G, {\bf A} )\to B _{\cH} (G, {\bf A})$ denote the universal $G$-equivariant vector bundle with fiber type ${\bf A}$ (see \cite[def. 2.4]{unlu-yalcin2}). Since $B_{\cH} (G, {\bf A} )^H=BC_{U(n)} (\rho_H)$ is simply connected for all $H \in \cH$, by standard obstruction theory there is a $G$-map $B^{(2)}\to B_{\cH} (G, {\bf A})$ (see the proof of Theorem 4.3 in \cite{unlu-yalcin2} for details). Pulling back the universal $G$-equivariant bundle via this map, we obtain a $G$-equivariant vector bundle over $B^{(2)}$. The sphere bundle of this bundle is spherical $G$-fibration over $B^{(2)}$ with fiber type $\{\bbS (V_H )\}$. Now the result follows from the repeated application of Proposition \ref{pro:mainconsttool}.
\end{proof}

We often want the total space of a $G$-fibration to be $G$-homotopy equivalent to a finite $G$-CW-complex. The following theorem gives a very useful criteria for this condition:

\begin{proposition}\label{pro:finiteness} Let $G$ be a finite group,  $B$ be a finite
$G$-CW-complex, and $p : E\to B$ be a $G$-fibration with fiber type $\{ F_H \}$. If $F_H$ is $H$-homotopy equivalent to a finite $H$-CW-complex for every $H \in \Iso (B)$, then $E$ is $G$-homotopy equivalent to a finite $G$-CW-complex.
\end{proposition}

\begin{proof} We will prove this lemma by induction over the skeletons of $B$. We already know that $p^{-1}(B^{(0)})$ is $G$-homotopy equivalent to a finite $G$-CW-complex. Now assume that $p^{-1}(B^{(n)})$ is
$G$-homotopy equivalent to a finite $G$-CW-complex $Z$ for some $n\geq 0$. We want to show that $p^{-1}(B^{(n+1)})$ is $G$-homotopy equivalent to a finite $G$-CW-complex.   

The pushout diagram given in $(1)$ induces a diagram of $G$-spaces  
\begin{equation}\label{eqn:pushout2}
\vcenter{\xymatrix{f^*(E_n)
\ar[r]^-{\overline f}\ar[d]^{\overline j} & E_n  \ar[d]^{\overline J} \\ g^*  (E_{n+1})
\ar[r]^-{\overline g} & E_{n+1}}}\end{equation}
where the spaces in the diagram are the total spaces of the fibrations obtained by taking pullbacks of the fibration $q_{n+1}: E_{n+1} \to B^{(n+1)}$ via the maps $f$, $g$, $j$, and $J$. Here $f=\coprod f_i$, $g=\coprod g_i$, $$j: 
\coprod _{i\in I_{n+1}} G/H_i\times \bbS ^{n}\to \coprod _{i\in I_{n+1}} G/H_i\times \bbD ^{n+1}$$
is the disjoint union of inclusion maps, and $J: B^{(n)}\to B^{(n+1)}$ is the inclusion map. Since the inclusion map $\bbS ^n \to \bbD^{n+1}$ is a cofibration map, the $G$-map $j$ is a $G$-cofibration.
So, by \cite[Lemma 1.26]{lueck}, the diagram \eqref{eqn:pushout2} is a pushout diagram and $\overline j$ is a $G$-cofibration.

Since $\bbD ^{n+1}$ is contractible, there is a $G$-fiber
homotopy equivalence
$$\displaystyle \coprod _{i\in I_{n+1}} G\times_{H_i} F_{H_i} \times\bbD ^{n+1} \maprt{\gamma } g^* (E_{n+1}).$$
This gives a commutative diagram of the following form
$$\xymatrix{E_{n}& & f^* (E_n)  \ar[ll]_{\overline f}  \ar[rr]^{\overline j}  & & g^* (E_{n+1})  \\ 
E_n \ar[u]_{\id} & & \displaystyle \coprod _{i\in I_{n+1}} G\times_{H_i} F_{H_i} \times \bbS ^{n}\ar[u]_-{\gamma '} \ar[rr]^{\id \times j} \ar[ll]_-{\overline f \circ \gamma '} & & \displaystyle \coprod
_{i\in I_{n+1}} G\times_{H_i} F_{H_i} \times \bbD ^{n+1} \ar[u]_-{\gamma}}$$ where $\gamma '$ is the restriction of $\gamma$ to the boundary spheres.
Such a restriction makes sense since $\gamma$ is a fiber homotopy equivalence.  
Now, since both $\id \times j$ and $\overline j$ are $G$-cofibrations, by \cite[Lemma 2.13]{lueck}, the $G$-space $E_{n+1}$, which is the pushout of the the diagram in the first line, is $G$-homotopy equivalent to the pushout of the diagram in the second line. 

To find a further homotopy equivalence, note that by induction assumption $E_n$ is $G$-homotopy equivalent to a finite $G$-CW-complex $Z$. So, using a similar diagram as above, we can conclude that $E_{n+1}$ is $G$-homotopy equivalent to the pushout of a diagram of the following form
$$\xymatrix{Z & & \displaystyle \coprod _{i\in I_{n+1}} G\times_{H_i} F_{H_i} \times \bbS ^{n}  \ar[rr]^{\id \times j} \ar[ll]_-{\varphi} & & \displaystyle \coprod
_{i\in I_{n+1}} G\times_{H_i} F_{H_i} \times \bbD ^{n+1} }.$$ 
Now, we can replace the map $\varphi$ with a cellular one (up to homotopy) and conclude that $E_{n+1}$ is $G$-homotopy equivalent to a finite $G$-CW-complex since the spaces $Z$ and $F_{H_i}$ for all $i\in I_{n+1}$ are finite $G$-CW-complexes. This completes the $n$-th stage of our induction. Since $B$ is a finite $G$-CW-complex, the induction will stop in finite steps. So, the proof is complete.
\end{proof}

\section{Proof of the main theorem}
\label{sect:mainthm}

Now, we are ready to prove the main theorem of the paper. First we introduce some notation and recall some basic facts about Stiefel manifolds. For more details, we refer the reader to \cite{unlu-yalcin1}.

Let $F$ denote the field of real numbers
$\bbR$, complex numbers $\bbC$, or quaternions $\bbH$. For a real
number the conjugation is defined by $\overline x =x$, for a complex
number $x=a+ib$ by $\overline x=a-ib$, and for a quaternion
$x=a+ib+jc+kd$ by $\overline x=a-ib-jc-kd$. On the vector space
$F^n$, we can define an inner product $(v,w)$ by taking
$$(v,w)=v_1 \overline w_1+ v_2 \overline w_2 + \cdots + v_n \overline w_n.$$ The Stiefel manifold
$V_k (F ^n )$ is defined as the subspace of $F^{nk}$ formed
by the $k$-tuples of vectors $(v_1, v_2, \dots, v_k)$ such that $v_i
\in F^n$ and for every pair $(i,j)$, we have
$(v_i , v_j )=1$ if $i=j$ and zero otherwise. 

There is a sequence of fiber bundles
$$V_n (F^n)\to \cdots \to V_{k+1} (F^n ) \maprt{q_k} V_k (F^n) \to
\cdots \to V_2 (F^n) \maprt{q_1} V_1 (F^n ) $$ where the map $q_k: V_{k+1}
(F^n ) \to V_{k} (F^n) $ is defined by $q_k(v_1, \dots, v_{k+1}
)=(v_1,\dots, v_k )$ and the fiber of $q_k$ is $V_1
(F^{n-k})=\bbS^{c(n-k)-1}$ where $c=\dim _{\bbR} F $
(see Theorem 3.8
and Corollary 3.9 in Chapter 8 of \cite{husemoller}).
Note that the
sphere bundle $q_k: V_{k+1}(F^n) \to V_k (F^n )$ is the sphere
bundle of the vector bundle $\overline q_k: \overline V_{k+1} (F^n )
\to V_k (F^n)$ where $\overline V_{k+1} (F^n)$ is the space formed
by $(k+1)$-tuples $(v_1, \dots, v_{k+1})$ satisfying $(v_1,\dots, v_k)\in
V_k (F^n)$ and $(v_i, v_{k+1})=0$ for all $i=1, \dots, k$.

Note that if a finite group $G$ has a representation $W$ over a field $F$, then the inner product above can be replaced by a $G$-invariant one and the Stiefel manifolds have natural $G$-actions. Moreover the sphere bundles given above become $G$-equivariant bundles. If the representation $W$ has fixity $f$, then we have $\dim _{F} W^g \leq f$ for all $g \in G$. This means if $W$ is a faithful representation, i.e., $f \leq \dim _F W-1$, then the $G$-action on $V_{f+1} (W)$ is free.
We will be using this observation in the proof of the main theorem.

Also observe that if $G$ has a complex representation with fixity $f$, then by tensoring it with $\bbH$ over $\bbC$, we obtain a symplectic representation with the same fixity. So to prove Theorem \ref{thm:main}, it is enough to prove the following:

\begin{theorem}\label{thm:mainspversion}
Let $\rho: G\to Sp(n)$ be a faithful symplectic representation with fixity $f$. Then there exists a finite $G$-CW-complex $X$ homotopy equivalent to a product of $f+1$ spheres such that the $G$ action the homology of $X$ is trivial.
\end{theorem}

\begin{proof} Let $W$ denote the $\bbH$-space corresponding to the representation $\rho$. Define $X_1=V_1(W)$. We will construct finite $G$-CW-complexes $X_2,
X_3,\dots,X_{f+1}$ recursively. For all $i$, the $G$-CW-complex $X_i$ will be homotopy equivalent to a product of $i$ spheres and will satisfy the following property: If $H\in \Iso(X_i)$, then $V_i(W)^H \neq \emptyset$ where $\Iso(X_i)$ denotes the set of isotropy subgroups of $X_i$. 

Assume that $X_i$ is constructed
for some $i\geq 1$. Note that for every $H \leq G$, the fixed point set $V_i(W)^H$ is either empty or simply connected. Since for every $H \in \Iso (X_i)$ we have $V(W_i )^H \neq \emptyset$, by standard equivariant obstruction theory there exists a $G$-map $f: X^{(2)}_i \to V_i(W)$.
By pulling back the $G$-bundle $q_i : V_{i+1} (W) \to V_i (W)$ via $f$, we obtain a $G$-equivariant vector bundle over $X_i^{(2)}$ with fiber type $\{ S(V_H )\}$. Note that this is a compatible family defined over all $H \in \Iso (X_i)$ where $V_H$ is the $H$-space $(\overline q_i) ^{-1} (b)$ for some $b \in V_i (W)^H$. 

Now, applying Proposition \ref{pro:mainconsttool} repeatedly to this $G$-fibration, we obtain a spherical $G$-fibration $E_i \to X_i$ with fiber type $\{ S(V_H ^{\oplus k})\}$ for some $k \geq 1$. By taking further fiber
joins, we can assume that $E_n$ is a trivial fibration non-equivariantly and the action on the homology of the total space is trivial. This is shown below in Lemma \ref{lem:homtrivial}. Now, by Proposition \ref{pro:finiteness},  $E_i$
is $G$-homotopy equivalent to
a finite $G$-CW-complex $Y$. Hence we can take $X_{i+1}$ as $Y$ and continue the induction until we reach $X_{f+1}$. At this stage, we have $(V_{f+1} (W))^H \neq \emptyset$ implies $H=\{1\}$, so we can conclude that the $G$-action on $X_{f+1}$ is free.
\end{proof}

\begin{lemma}\label{lem:homtrivial} Let $p: E \to B$ be a $G$-fibration over a finite $G$-CW-complex $B$ and $n$ be a positive integer. Suppose that $p$ has fiber type $\{ F_H \}$ such that $F_H $ is homotopy equivalent to the sphere $\bbS^n$ for all $H \in \Iso (B)$. Then, there is an integer $k\geq 1$ such that $\ast _k p : \ast _k E \to B$ is non-equivariantly homotopy equivalent to the trivial fibration. Moreover, if the $G$-action on the cohomology of $B$ is  trivial, then we can choose $k$ large enough so that the $G$-action on the cohomology of $E$ is also trivial. 
\end{lemma}

\begin{proof} The first part of the lemma is well-known and it follows from the fact that the homotopy groups of $\Aut (S^n)$
is finite. For the second part, observe that since the resulting fibration is homotopy equivalent to the trivial fibration, it is in particular an orientable fibration, i.e., $\pi _1 (B)$ action on the homology of $F$ is trivial. So, there exists a consistent choice of generators for $H^n (F_b )\cong \bbZ$ for all $b \in B$. Note that this gives a $G$-action on the cohomology of the fibers $H^n (F_b)$ which is defined by $$g^* : H^n (F_b) \to H^n (F_{gb})\cong H^n (F_b)$$ where the isomorphism on the right comes from identifications of generators that we have chosen. Observe that this action in general can be nontrivial since a generator $u$ can go to $-u$ but if we take the fiber join of $p$ with itself, then we can assume that this action is trivial for all $b \in B$.

For an orientable spherical fibration there is a Serre spectral sequence   $$E_2 ^{p,q} =H^p (B, H^q (F )) \Longrightarrow H^{p+q} (E).$$
Note that for a $G$-fibration, all the terms in this spectral 
sequence will be $\bbZ G$-modules and the differentials will be $\bbZ G$-module homomorphisms since every $g \in G$ induces a fiber preserving continuous map $g: E \to E$. In our case, we have a two line spectral sequence and it is easy to see that by choosing $k$ large enough, we can assume that $N=\dim (\ast _k \bbS ^n) \geq \dim B$, and hence we can conclude that $H^i (E) \cong H^i (B)$ for $i<N$ and $H^i (E) \cong H^{i-N} (B, H^N( F) )$ for $i \geq N$. Since $G$ acts trivially on $H^* (B)$, we have trivial action on $H^*(E)$ if the $G$-action on $H^{i-N} (B, H^N(F))$ is trivial for all $i\geq N$. Note that the Serre spectral sequence has a product structure, so the action on $H^* (B, H^N (F))$ is trivial if the $G$-action on $H^0 (B, H^N (F))$ is trivial. Note that $H^0 (B, H^n (F))$ is the kernel of the map
$$ d^1 : \Hom _{\bbZ} (C_0 (B), H^N (F)) \to \Hom _{\bbZ} (C_1 (B), H^N (F))$$
and as a $\bbZ G$-module 
$$\Hom _{\bbZ } (C_i (B), H^N (F)) \cong \bigoplus \limits _{\sigma \in I_i} \Hom _{\bbZ} (\ind _{G_{\sigma}} ^G \bbZ _{\sigma}, H^N (F_{\sigma} ))$$
for $i=0,1$ where the $G$-action on $H^N(F_{\sigma})$ is the one described above. Since we assumed that this action is trivial, we can conclude that $H^0 (B, H^N (F) )\cong H^0 (B)$ as $\bbZ G$-modules, and hence $H^0 (B, H^{N} (F))$  is a trivial $\bbZ G$-module.
This completes the proof of the lemma.

\end{proof}

\acknowledgement{We thank Natalia Castellana
for directing us towards M\o ller's work on equivariant Federer spectral sequence for calculating the homotopy groups of the monoid $\Aut _H (F_H)$.}

\providecommand{\bysame}{\leavevmode\hbox
to3em{\hrulefill}\thinspace}
\providecommand{\MR}{\relax\ifhmode\unskip\space\fi MR }
\providecommand{\MRhref}[2]{%
  \href{http://www.ams.org/mathscinet-getitem?mr=#1}{#2}
} \providecommand{\href}[2]{#2}

\end{document}